\documentclass[11pt]{article}

\usepackage[margin=1in]{geometry}

\usepackage{float}
\usepackage{graphicx}
\usepackage{amsfonts,amsmath,amssymb,amsthm}
\usepackage{caption}
\usepackage{algorithm}
\usepackage{algpseudocode}
\usepackage{graphicx}

\newtheorem{theorem}{Theorem}[section]
\newtheorem{proposition}[theorem]{Proposition}
\newtheorem{lemma}[theorem]{Lemma}
\newtheorem{corollary}[theorem]{Corollary}
\newtheorem{definition}[theorem]{Definition}
\newtheorem{observation}[theorem]{Observation}

\usepackage{tikz}
\usetikzlibrary{arrows,positioning,trees,matrix}\pdfpageattr {/Group << /S /Transparency /I true /CS /DeviceRGB>>}
\tikzset{vertex/.style={circle,draw=black,fill=black!20,thin,node distance=0.75}}
\tikzset{edge/.style={thick}}

\begin{document}

\title{Bounds and algorithms for graph trusses}
\author{
{\sc Paul Burkhardt}$^{a}$
\and
{\sc Vance Faber}$^{b}$
\and
{\sc David G. Harris}$^{c}$
}

\renewcommand{\thefootnote}{a} 
\footnotetext{U.S.~National Security Agency, Ft.~Meade, MD 20755. 
Email: \texttt{pburkha@nsa.gov}}

\renewcommand{\thefootnote}{b} 
\footnotetext{IDA/Center for Computing Sciences, Bowie MD 20707. Email: \texttt{vance.faber@gmail.com}}

\renewcommand{\thefootnote}{c} 
\footnotetext{Department of Computer Science, University of Maryland, 
College Park, MD 20742. 
Email: \texttt{davidgharris29@gmail.com}}

\date{\today}

\maketitle 

\begin{abstract}
The $k$-truss, introduced by Cohen (2005), is a graph where every edge is incident to at least $k$
triangles. This is a relaxation of the clique. It has proved to be a useful tool in identifying cohesive
subnetworks in a variety of real-world graphs. Despite its simplicity and its utility, the combinatorial and algorithmic aspects of trusses have not been thoroughly explored. 

We provide nearly-tight bounds on the edge counts of $k$-trusses. We also give two improved algorithms for finding trusses in large-scale
graphs. First, we present a simplified and faster algorithm, based on approach
discussed in Wang \& Cheng (2012). Second, we present a theoretical algorithm based on
fast matrix multiplication; this converts a triangle-generation algorithm of Bj{\"o}rklund et al. (2014) into a dynamic
data structure.

\end{abstract}

\section{Introduction}
In a number of contexts, a group of interacting agents can be represented in terms of an undirected graph $G = (V,E)$.  For example, in a social network, the vertices may represent people with an edge if they know each other. One basic task is to find 
a \emph{cohesive subnetwork} of $G$: a maximal subgraph whose vertices are ``highly connected'' \cite{bib:seidman1983}. This may represent a discrete community in the overall network, or another type of subgroup with a high degree of mutual relationship. We emphasize that since we are ultimately trying to understand a non-mathematical property of $G$, we cannot give an exact definition of a cohesive subnetwork.

A number of graph-theoretic structures can be used to find cohesive subnetworks in $G$. A clique is the most highly connected substructure. An alternate choice, suggested by \cite{bib:seidman1983}, is the \emph{$k$-core}, which is defined as a maximal connected subgraph in which each vertex has degree at least $k$. 

Cohen \cite{tech-report-cohen, bib:cohen2009} proposed a stronger heuristic based on triangle counts called the \emph{truss}. Formally, a \emph{$k$-truss} is defined to be a graph in which every edge is incident to at least $k$ triangles and which has no isolated vertices. Note that a $(k+2)$-clique is a $k$-truss.\renewcommand{\thefootnote}{1}\footnote{Cohen defined the $k$-truss as being a connected graph
  such that every edge is incident to at least $k-2$
  triangles. This was presumably chosen so that
  a $k$-clique is a $k$-truss.}  A \emph{$k$-truss-component of $G$} is defined to be a maximal edge set $L \subseteq E$ such that the edge-induced subgraph $G(L)$ is a connected $k$-truss.

The $k$-truss has been rediscovered and renamed several times. The earliest example was its definition as a \emph{$k$-dense}
core~\cite{bib:saito2006}, which was motivated by the goal of detecting dense
communities where the $k$-core proved to be too coarse. It was also defined as a
\emph{triangle $k$-core} in~\cite{bib:zhang2012} and used as a motif exemplar in
graphs. Other names include \emph{$k$-community}~\cite{bib:verma2013} and
\emph{$k$-brace}~\cite{bib:ugander2012}.

The task of determining the truss-components of $G$ is called \emph{truss decomposition.} The truss-components of $G$ can be derived from an associated
hypergraph $H$ which is defined as follows: the vertex set of $H$ is the edge set of $G$, and the edge set of $H$ is the set of triangles of $G$.  Each $k$-core of $H$ corresponds to a $k$-truss-component of $G$.

The \emph{trussness} of an edge $e$ of $G$, denoted $\tau(e)$, is defined to be the maximal value $k$ such that $e$ is in a $k$-truss-component of $G$. Equivalently, $\tau(e)$ is the coreness of the edge $e$ regarded as a node of $H$. Truss decomposition algorithms typically first compute $\tau(e)$ for all edges $e$. The $k$-truss-components (for any value of $k$) can then be found by depth-first search of $G$ restricted to edges $e$ with $\tau(e) \geq k$.

An appealing feature of trusses is that truss decomposition algorithms are relatively practical, making them feasible for large graphs.  As a starting point, Cohen's original algorithm \cite{bib:cohen2009} was essentially an adaptation of a graph core decomposition algorithm of Matula \& Beck \cite{matula} to the hypergraph $H$. This was improved by Wang \& Cheng \cite{bib:wangcheng2012} and Huang et al. \cite{huang2015approximate} by avoiding explicit generation of $H$. These algorithms compute  truss decomposition in $O(m \alpha(G))$ time and $O(m)$ memory, where $\alpha(G)$ denotes the arboricity of $G$. Note that $m \alpha(G)$ is at most $O(m^{3/2})$.

The $k$-truss-component structure of graphs has become a key tool for pattern mining and community detection  in a variety of scientific and social
network studies, e.g. \cite{bib:ugander2012, bib:cheng2014}. Additionally, the truss serves as a fast filter for finding cliques since a $(k+2)$-clique is a $k$-truss. In these applications, the parameter $k$ represents the anticipated size of a meaningful cohesive subnetwork. Typically, $k$ may be much smaller than $n$. In this setting, it is often useful to compute a \emph{truncated} truss decomposition up to some chosen parameter $k_{\text{trunc}}$. This entails computing $\tau(e)$ for edges $e$ with $\tau(e) \leq k_{\text{trunc}}$; other edges $e$  record that $\tau(e) > k_{\text{trunc}}$ but do not record the precise value. For instance, Verma et al. \cite{verma2015solving} discussed an algorithm for sparse graphs with runtime $O(m k_{\text{trunc}} \Delta)$, where $\Delta$ is the maximum degree of $G$. 

\subsection{Our contributions and overview}
Despite its simplicity and its use for understanding real-world networks, the truss has seen little formal analysis, either from a combinatorial or algorithmic point of view. We address these gaps in this paper.

Our main combinatorial subject of investigation is the minimum number of edges in a $k$-truss. Section~\ref{sec:min_edgecount} gives asymptotically tight bounds for this quantity. Section~\ref{sec:critical_connectivity} analyzes a stricter notion of \emph{critical $k$-truss}, which is a $k$-truss none of whose subgraphs are themselves $k$-trusses. We summarize the main results of these sections as follows:
\begin{theorem}
  The minimum number of edges in a connected $k$-truss on $n$ vertices, is $n (1 + k/2) + \Theta(k^2)$. The minimum number of edges in a critical $k$-truss on $n$ vertices is $n k/2 + \Theta(n + k^2)$.
\end{theorem}

Section~\ref{sec:algorithm1} describes a new simple algorithm for truss decomposition. We analyze this algorithm in terms of a graph parameter we refer to as the \emph{average degeneracy $\bar \delta(G)$}, defined as:
$$
\bar \delta(G) = \frac{1}{m} \sum_{\substack{\text{edges} \\ e = (u,v)}}  \min( \deg(u), \deg(v)).
$$

We have $\bar \delta(G) \leq 2 \alpha(G) \leq O(\sqrt{m})$. While $\alpha(G)$ is a local property of the graph, which can be influenced by a few  high-degree nodes, the parameter $\bar \delta(G)$ is a global property and can be much smaller than $\alpha(G)$. To the best of our knowledge, this parameter has not been studied before.  We show the following result:
\begin{theorem}
  There is an algorithm to compute the truss decomposition of $G$ in $O(m \bar \delta(G)) \leq O(m \alpha(G)) \leq O(m^{3/2})$ time and $O(m)$ memory.
\end{theorem}
This algorithm is inspired by Wang \& Cheng \cite{bib:wangcheng2012} but uses much simpler and faster data structures; it should be practical for large-scale graphs.  

Section~\ref{sec:algorithm2} gives an alternative  algorithm for truncated truss decomposition based on 
matrix multiplication. We present here a slightly simplified summary in terms of the linear algebra constant $\omega$, i.e. the value for which multiplication of $N \times N$ matrices can be performed in $N^{\omega + o(1)}$ time. (The current best estimate \cite{bib:legall2014} is $\omega \approx 2.38$.) The algorithm is theoretically appealing, but the algorithm in
Section~\ref{sec:algorithm1} is more likely to be useful in practice.

\begin{theorem}
  There is an algorithm to compute the truncated truss decomposition up to any desired parameter $k_{\text{trunc}}$ in $m^{\frac{2 \omega}{\omega+1} + o(1)} k_{\text{trunc}}^{\frac{1}{\omega+1}}$ time and $m^{\frac{4}{\omega+1} + o(1)} + m^{1+o(1)} k_{\text{trunc}}$ memory.
\end{theorem}

Fast matrix multiplication has been used for previous triangle counting and enumeration algorithms \cite{bib:ayz1997, bib:bjorklund2014}. For our decomposition algorithm, we turn these into dynamic data structures to maintain triangle lists as edges of $G$ are removed.   See Section~\ref{sec:algorithm2} for more precise costing.

\subsection{Truss combinatorics in the context of cohesive subnetworks}
The truss is an interesting but somewhat obscure combinatorial object, and this is the primary reason for studying its combinatorial properties. In addition, there is an important connection between the extremal bounds for trusses and its use as a heuristic for cohesive subnetworks.

Intuitively, a cohesive subnetwork on $n$ nodes should be highly connected. Thus the edge counts should be quite high, perhaps on the order of $n^2$. By contrast, our results in Section~\ref{sec:min_edgecount} show an example of a connected $k$-truss with edge counts as low as  $\Theta(k n)$.  It may be surprising that a connected $k$-truss can be extremely sparse despite its use in finding cohesive networks. These can be viewed as pathological cases where the truss heuristic does a poor job at discovering the underlying graph structure.

The extremal example consists of a series of $(k+2)$-cliques connected at vertices. This is clearly a collection of multiple distinct cohesive subnetworks. In particular, it contains many subgraphs which are themselves $k$-trusses.  This extremal example motivates us to define a stricter notion of \emph{critical $k$-truss} as a heuristic for finding cohesive subnetworks: namely, a collection of edges which is a $k$-truss but which contains no smaller $k$-truss.

It seems reasonable that this restriction might give a more robust heuristic for cohesive subnetworks. Yet, we will show in Section~\ref{sec:critical_connectivity} that this has similar extremal examples.  The additional restriction of criticality does not significantly increase the minimum edge count. These results suggest that, despite their usefulness for real-world graphs, both the $k$-truss and the critical $k$-truss can be fallible heuristics.

\subsection{Notation}
We let $n$ denote the number of vertices and $m$ the number of edges of a graph $G = (V,E)$. The neighborhood of a vertex $v \in V$
is the set $N(v) = \{u : (u,v) \in E\}$ and $d(v) = \lvert N(v) \rvert$ is the degree of $v$. We also define $N^+(v) = N(v) \cup \{v \}$. For simplicity, we assume throughout that $G$ has no isolated vertices and $m \geq n/2$.

We define a triangle to be a set of three vertices $(v_1, v_2, v_3)$ where edges $(v_1, v_2), (v_2, v_3), (v_3, v_1)$ are all present in $G$. We also write $(e_1, e_2, e_3)$ for this triangle, where edges $e_1, e_2, e_3$ are given by $e_1 = (v_1, v_2), e_2 = (v_2, v_3), e_3 = (v_3, v_1)$. We define $\triangle(e)$ and $\triangle(v)$ to be the number of triangles containing an edge $e$ or vertex $v$ respectively.  We say edge $e_1$ and $e_2$ are neighbors if they share a vertex.

For an integer $t$ we define $[t] =  \{1, \dots, t \}$. We assume there is some fixed, but arbitrary, indexing of the vertices, and we define $\text{ID}(v) \in [n]$ to be the identifier of vertex $v$.

For a vertex subset $U \subseteq V$, we define the \emph{vertex-induced subgraph} $G[U]$ to be the graph on vertex set $U$ and edge set $\{ (u,v) \in E : u,v  \in U \}$. For an edge set $L \subseteq E$, we define the \emph{edge-induced subgraph} $G(L)$ to be the graph on edge set $L$ and vertex set $\{ v \in V : \text{$(u,v) \in L$ for some $u \in V$} \}$.
 
The complete graph on $n$ vertices ($n$-clique) is denoted by
$K_n$.

We will analyze some algorithms in terms of the \emph{degeneracy} of graph $G$, which we denote $\delta(G)$. See Appendix~\ref{app:arb} for further definitions and properties.

\section{Minimum edge counts for the $k$-truss}
\label{sec:min_edgecount}

We begin by collecting a few simple observations on the vertex counts in a $k$-truss.

\begin{observation}
  \label{obs:mintriangle}
Any vertex $v$ in a $k$-truss $G$ must have degree at least $k+1$. Furthermore, the graph $G[N^+(v)]$ has at least $\binom{k+1}{2}$ triangles and $\binom{k+2}{2}$ edges.  
\end{observation}
\begin{proof}
Let $e$ be any edge on $v$. This edge $e$ has at least $k$ triangles, giving $k$ other edges incident on $v$. Each of these $k+1$ edges has at least $k$ triangles in $G[N^+(v)]$; furthermore, each such triangle is counted by at most two edges incident on $v$. So $\triangle(v)$ is at least $k (k+1) / 2$. Finally, the number of edges in $G[N^+(v)]$ is precisely $d(v) + \triangle(v)$, which is at least $(k+1) + \binom{k+1}{2} = \binom{k+2}{2}$.
  \end{proof}

\begin{observation}
\label{obs:minsize}
The minimum number of vertices in a $k$-truss is exactly $k+2$.
\end{observation}
\begin{proof}
Each vertex must have degree at least $k+1$, thus each $k$-truss contains at
least $k+2$ vertices. On the other hand, $K_{k+2}$ is a $k$-truss with
$k+2$ vertices.
\end{proof}

The properties can be used to bound the clustering coefficient, a common measure of graph density. Formally, the \emph{clustering coefficient} of a vertex $v$ is defined as $cc(v) = \triangle(v) / \tbinom{d(v)}{2}$. Thus,  $cc(v)$ is a real number in the range $[0,1]$.

\begin{observation}
  \label{cc-obs}
In a $k$-truss, each vertex $v$ has clustering coefficient $cc(v) \geq \binom{k+1}{2}/\binom{d(v)}{2}$. In particular,  if $d(v) = k+1$, then  $cc(v) = 1$ and $G[N^+(v)]$ is a $(k+2)$-clique.
\end{observation}

We also get a simple bound on the edge counts.
\begin{observation}
  \label{tt-bcor}
  If graph $G$ has $m$ edges, then $\tau(e) \leq \sqrt{2 m + 1/4} - 3/2 \leq \sqrt{2 m}$ for all edges $e$.
\end{observation}
\begin{proof}
  Suppose that edge $e = (u,v)$ has $\tau(e) = k$. Let $G'$ denote the $k$-truss component containing edge $e$. By Observation~\ref{obs:mintriangle}, $G'[N^+(v)]$ contains at least $\binom{k+2}{2}$ edges, and hence $\binom{k+2}{2} \leq m$.
\end{proof}

Let us define $M_{n,k}$ to be the minimum number of edges for a \emph{connected} $k$-truss on $n$ vertices. Our main result in this section is to estimate this quantity $M_{n,k}$, showing the following tight bounds:
\begin{theorem}
\label{thm:edge_bounds}
For every $k \geq 1$ and $n \geq k+2$, we have
$$
(n-1) (1 + k/2) \leq M_{n,k} \leq n (1 + k/2) + \Theta(k^2).
$$

Furthermore, if $n \equiv 1 \mod k+1$, then $M_{n,k} = (n-1) (1 + k/2)$.
\end{theorem}

As an immediate corollary of Theorem~\ref{thm:edge_bounds}, we also get a tight bound on triangle counts:
\begin{corollary}
\label{thm:triangle_bounds}
A connected $k$-truss on $n$ vertices must contain at least $\frac{(n-1)(k+2) k}{6}$ triangles. 
\end{corollary}
\begin{proof}
  The graph has at least $m \geq (n-1)(1 + k/2)$ edges. Each edge has at least $k$ triangles; since each triangle is incident to three edges, this implies there are at least $m k/3$ triangles. 
\end{proof}

For the upper bound of Theorem~\ref{thm:edge_bounds}, we use a construction based on vertex contraction.  Namely, for a pair of graphs $G, H$,  define $G * H$ to be the graph obtained by contracting an arbitrary vertex of $G$ to an arbitrary vertex of $H$. The resulting graph $G*H$ has $|V(G)| + |V(H)| - 1$ vertices and $|E(G)| + |E(H)|$ edges.

Now when $n \equiv 1 \mod k+1$, we get the upper bound $M_{n,k} \leq (n-1) (1 + k/2)$ by taking $G = A_1 * \dots * A_s$  where $A_1, \dots, A_s$ are copies of $K_{k+2}$ and $s = \frac{n-1}{k+1}$. This also shows that the bound of Corollary~\ref{thm:triangle_bounds} is tight in this case.

A slightly modified construction shows the upper bound $M_{n,k} \leq n(1+k/2) + O(k^2)$ for arbitrary values of $n$.  Namely, let $n = s (k+1) + r$ for integer $r$ in the range $k+2 \leq r < 2 k + 3$, and consider $G = A_1 * \dots * A_s * B$, where $A_1, \dots, A_s$ are copies of $K_{k+2}$ and $B$ is a copy of $K_r$. Noting that $r \leq 2k$, we calculate the edge count of $G$ as $m = s \binom{k+2}{2} + \binom{r}{2} \leq n ( 1 + k/2) + O(k^2)$.

We now turn to prove the lower bound of Theorem~\ref{thm:edge_bounds}. Let $G$ be a connected $k$-truss. Since $G$ is connected, it has a spanning tree $T$, which we may take to be a rooted tree (with an arbitrary root).  For a triangle $t$ of $G$, we say that $t$ is \emph{single-tree} if exactly one edge is in $T$, otherwise it is \emph{double-tree}. (If all three edges were in $T$, this would be a cycle on $T$.)  If an edge $e \in E - T$ participates in a double-tree triangle, we say that $e$ is \emph{double-tree-compatible} otherwise it is \emph{double-tree-incompatible}.

\begin{proposition}
\label{me1}
Let $(u,v,w)$ be a single-tree triangle where $( u, v ) \in T$. Then either $(u,w)$ or $(v,w)$ is double-tree-incompatible.
\end{proposition}
\begin{proof}
  If $(u,w), (v,w)$ are both double-tree-compatible, this would imply that there are vertices $x, y$ with $( u, x ), ( x, w ), ( v, y ), ( w, y ) \in T$. Since $(u,w) \in E - T$, we must have $y \neq u$.  But then $u, x, w, y, v, u$ is a cycle on the tree $T$, which is a contradiction.
\end{proof}

Our proof strategy will be to construct a function $F: T \times [k] \rightarrow E - T$, and then argue that $F$ is $2$-to-$1$.  This will show that $E - T$ has cardinality at least $|T| k/2$, and so $|E| \geq |T| + |T| k/2 = (n-1) (1+k/2)$ which is the lower bound we need to show.

To define $F$, consider an edge $e = ( x, y ) \in T$, where $y$ is a $T$-child of $x$. Arbitrarily select $k$ triangles $t_{e,1}, \dots, t_{e,k}$ involving $e$.  For $i = 1, \dots, k$, we define $F(e,i)$ as follows:
\begin{itemize}
\item If $t_{e,i}$ is double-tree, then $F(e,i)$ is the unique off-tree edge of $t_{e,i}$.
\item If $t_{e,i}$ is single-tree and exactly one off-tree edge $f$ of $t_{e,i}$ is double-tree-incompatible, then $F(e,i) = f$.
\item If $t_{e,i}$ is single-tree and both off-tree edges of $t_{e,i}$ are double-tree-incompatible,  then $F(e,i)$ is the off-tree edge of $t_{e,i}$ containing $y$.
\end{itemize}

In light of Proposition~\ref{me1}, this fully defines the function $F$.  We refer to an edge $e$ as a preimage of $f$ if $F(e,i) = f$ for some index $i \in [k]$. Since any triangle is determined by two of its edges, such index $i$ is uniquely determined by $e$ and $f$.

\begin{proposition}
  \label{me3}
  Suppose that edge $f = (u,v) \in E - T$ is double-tree-incompatible, and $(u,x)$ is a preimage of $f$ where $x$ is a $T$-child of $u$.   Then edge $(x,v)$ must be double-tree-compatible. Furthermore, $f$ does not have a preimage $(u,y)$ where $y$ is a  $T$-child of $u$ distinct from $x$.
\end{proposition}
\begin{proof}
 For the first result, suppose that $F( (u,x), i) = f$ and consider the triangle $t_{(u,x),i}$. Since $f$ is double-tree-incompatible, necessarily $t_{(u,x),i}$ is single-tree. If the other edge $(x, v)$ in this triangle were also double-tree-incompatible, then since $x$ is a $T$-child of $u$ we would have $F( (u,x), i ) = (x, v) \neq f$, a contradiction.

 For the second result, suppose that $(u,x)$ and $(u,y)$ are preimages of $f$. By the argument in the preceding paragraph the edges $(x,v)$ and $(y,v)$ are both double-tree-compatible. So there are vertices $r,s$ with $( x, r ), ( r, v ), ( y, s ), ( s, v ) \in T$. One can then check that $v,s,y,u,x,r,v$ is a cycle on $T$, a contradiction.
\end{proof}

\begin{proposition}
Every edge $f = (u,v) \in E-T$ has at most two preimages under $F$.
\end{proposition}
\begin{proof}
  \textbf{Case I: $f$ is double-tree-compatible.}   The only possible preimages to $f$ would come from double-tree triangles in which $f$ is the unique off-tree edge. There can only be a single such triangle; for, if $(u,v,x)$ and $(u,v,y)$ were two such triangles, then $u,x,v,y,u$ would be a cycle on $T$. So the only possible preimages of $f$  are the two tree-edges in this double-tree triangle.

\noindent  \textbf{Case II: $f$ is double-tree-incompatible.} We first claim that $f$ cannot have three  preimages $(u,x ), (u, y ), ( u, z )$. For, in this case, at least two vertices, say without loss of generality $x,y$, must be $T$-children of $u$. This is ruled out by Proposition~\ref{me3}.

So suppose that $f$ has three preimages $( u, x ), (u, y), ( v, z )$. By Proposition~\ref{me3}, it cannot be that both $x$ and $y$ are $T$-children of $u$. So assume without loss of generality that $y$ is the $T$-parent of $u$ and $x$ is a $T$-child of $u$.

By Proposition~\ref{me3}, the edge $( x,v )$ is double-tree-compatible. So there is some vertex $r$ such that $( x, r ), ( r, v ) \in T$.  Since $u$ is the $T$-parent of $y$, this implies that $x$ is the $T$-parent of $r$ and $r$ is the $T$-parent of $v$ and $v$ is the $T$-parent of $z$.

Since $(v,z)$ is a preimage of $f$ and $v$ is the $T$-parent of $z$, by Proposition~\ref{me3} the edge $( z, u )$ must be double-tree-compatible. So we have $( z, s ), ( s, u ) \in T$ for some vertex $s$. It can be seen that  $u, x, r, v, z, s, u$ is a cycle on $T$, a contradiction.
\end{proof}

This completes the proof of the lower bound of Theorem~\ref{thm:edge_bounds}.

\section{Critical connectivity of the $k$-truss}
\label{sec:critical_connectivity}

We define a \emph{critical $k$-truss} as follows:

\begin{definition}[Critical $k$-truss]
Graph $G$ is a \emph{critical} $k$-truss if $G$ is a $k$-truss, but $G(E')$ is not a $k$-truss for any non-empty $E' \subsetneq E$.
\end{definition}

It is clear that critical $k$-trusses are connected. The extremal graphs $A_1 * \dots * A_s$ for the upper bound in Theorem~\ref{thm:edge_bounds} are far from critical, as each subgraph $A_i$ is a $k$-truss. Arguably, critical $k$-trusses are more relevant for community detection --- if a  graph contains smaller $k$-trusses, then it is a conglomeration of  communities rather than a single cohesive subnetwork of its own.

We begin with some simple observations.
\begin{observation}
  \label{oo1}
 There is no critical $k$-truss with exactly $k+3$ vertices. In a critical $k$-truss with more than $k+3$ vertices, every vertex has degree at least $k+2$.
\end{observation}
\begin{proof}
  By Observation~\ref{obs:mintriangle}, each vertex $v$ has $d(v) \geq k+1$. By Observation~\ref{cc-obs}, if $d(v) = k+1$, then the neighborhood of $v$ is a copy of $K_{k+2}$, which is impossible in a critical $k$-truss. Furthermore, if $n = k+3$ and $d(v) = k+2$ for all vertices $v$, then $G$ is a copy of $K_{k+3}$, and in particular $G$ contains a copy of the $k$-truss $K_{k+2}$.
\end{proof}

\begin{observation}
  \label{oo2}
The graph $K_3$ is the only critical $1$-truss.
\end{observation}
\begin{proof}
Suppose that $G$ is a critical $1$-truss, and let $e_1$ be an edge of $G$. So $e_1$ is contained in some triangle with edges $(e_1, e_2, e_3)$. Then $G( \{e_1, e_2, e_3 \})$ is a $1$-truss.
\end{proof}

Let us define $M^*_{n,k}$ to be the minimum number of edges in a critical $n$-node $k$-truss; we have $M^*_{n,k} = \infty$ if no such critical $k$-truss exists. To avoid edge cases covered by Observations~\ref{oo1} and \ref{oo2}, we assume that $k \geq 2$ and $n \geq k+4$.

For $k = 2$, we can compute the precise value of $M^*_{n,k}$:
\begin{theorem}
  \label{mstar2}
  For all $n \geq 6$ we have $M^*_{n,2} = 3 n - 6.$
  \end{theorem}
\begin{proof}
For  the upper bound, consider the graph consisting of a cycle $C$ of length $n-2$, plus two new vertices $x_1, x_2$ with edges to every vertex in $C$. This is a critical $2$-truss with $n$ vertices and $3 n - 6$ edges. 

For the lower bound, let $G = (V,E)$ be a critical $2$-truss with $m$ edges and $n$ vertices. Define $\mathcal V$ to be the vector space of all functions from $E$ to the finite field $GF(2)$, and for any triangle $t$ of $G$ we define $\chi_t$ to be the characteristic function of $t$, i.e. $\chi_t(e) = 1$ if and only if $e \in t$.

Select $U$ to be any smallest set of triangles in $G$ with the property that every edge is in at least two triangles of $U$. This is well-defined since $G$ is a $2$-truss. We claim that for every proper subset $W \subsetneq U$, the sum $\sum_{t \in W} \chi_t$ is not identically zero.

For, suppose it is, and let $L$ denote the set of edges appearing in the triangles $t \in W$. For any edge $e \in L$, we then have $\sum_{t \in W} \chi_t(e) = 0$ in the field $GF(2)$. Since the sum is taken modulo two, there are at least two triangles $t_1, t_2$ in $G(L)$ containing $e$ (by definition of $L$, there is at least one). Thus $G(L)$ is a $2$-truss. Since $G$ is a critical 2-truss, we must have $L = E$. Thus, every edge of $G$ is covered by at least two triangles in $W$. This contradicts minimality of $U$.

We have shown that there is at most one linear dependency among the functions $\chi_t$ for $t \in U$ (namely, corresponding to $W = U$).  A standard result (see e.g. \cite{diestel2012graph}) is that the vector space $\mathcal V$ has dimension $m - n + 1$. This implies that $|U| - 1 \leq m - n + 1$. On the other hand, every edge is in at least two triangles in $U$ and so $3 |U| \geq 2 m$. Putting these inequalities together gives $m \geq 3 n-6$.
\end{proof}

Our main result in this section is to estimate $M^*_{n.k}$, showing a result similar to Theorem~\ref{thm:edge_bounds}. Specifically, we will show that
$$
M^*_{n,k} = ( \Theta(1) + k/2 ) n + \Theta(k^2).
$$
(A more precise estimate is shown in Theorem~\ref{mnk-thm}.)

\begin{lemma}
\label{sslem}
For any integers $k \geq 2, n \geq k+4$ we have
$$
M^*_{n+2,k+2} \leq M^*_{n,k} + 2 n, \qquad M^*_{n+1, k+1} \leq M^*_{n,k} + n.
$$
\end{lemma}
\begin{proof}
  To show the first bound, let $G = (V,E)$ be a critical $k$-truss with $n$ vertices and $m = M^*_{n,k}$ edges. Create a new graph $G'$, which has all the vertices and edges of $G$, plus two new vertices $x_1, x_2$. We add a new set $F$ of edges connecting $x_1, x_2$ to the previous vertices, where $F$ is chosen so that $G' = (V \cup \{x_1, x_2 \}, E \cup F)$ has the properties that (i) $G'$ is a $(k+2)$-truss and (ii) $F$ is inclusion-wise minimal with this property.

  To show this  is well-defined, we need to show that property (i) is satisfied when $F$ is the set of all possible edges between the new vertices and the old ones. In this case, $G'$ has no isolated vertices, as $G$ is a $k$-truss and has none. Also, any edge $e \in E$ has $k$ triangles from $G$ and two new triangles in $F$. Finally, for each edge $e = (x_i, v)$ where $v \in G$, there is a triangle in $G'$ for each neighbor of $v$ in $G$. By Observation~\ref{oo1}, this implies that $\triangle(e) \geq k+2$.
 
Clearly $G'$ has $n+2$ vertices and has $m + |F| \leq m + 2n$ edges. To show  $G'$ is critical, suppose there are edge subsets $E' \subseteq E, F' \subseteq F$ such that $G' (E' \cup F')$ is a $(k+2)$-truss. Then $G(E')$ must be a $k$-truss, as removing the edges incident to $x_1, x_2$ can only remove $2$ triangles per edge. Since $G$ is critical, this implies  $E' = \emptyset$ or $E' = E$. If $E' = \emptyset$, then $G(F')$ must be a $k$-truss. However, $G(F)$ itself has no triangles, and so we must have $F' = \emptyset$, i.e. $E' \cup F' = \emptyset$. On the other hand, if $E' = E$, then $G'(E \cup F')$ is a $(k+2)$-truss; by definition of $F$, this implies that $F' = F$ and hence $E' \cup F' = E \cup F$.

The second bound is essentially identical, except that we add only a single vertex instead of two vertices.
\end{proof}

\begin{proposition}
\label{hh1}
For $n \geq k+4$, we have $M^*_{n,k} \leq n(k+1) - k^2/2 - 2 k + 1/2$.
\end{proposition}
\begin{proof}
 Suppose first that $k$ is even. From Theorem~\ref{mstar2}, we have $M^*_{i,2} = 3 i - 6$ for any integer $i \geq 6$. By repeated applications of Lemma~\ref{sslem} we get
$$
M^*_{i+2j, 2+2j} \leq M^*_{i,2} + 2 i + 2 (i+2) + \dots + 2 (i + 2 j - 2) \leq 3 i - 6 + 2 j (j + i - 1).
$$

Setting $j = (k-2)/2$ and $i = n + 2 - k$ gives $M^*_{n,k} \leq n(k+1) - k^2/2 - 2 k$.

If $k$ is odd, then by Lemma~\ref{sslem} we have $M^*_{n, k} \leq M^*_{n-1,k-1} + (n-1)$. Since  $k-1$ is even, the argument of the preceding paragraph gives  $M^*_{n-1,k-1} \leq (n-1) k - (k-1)^2/2 - 2 (k-1)$, and so
\[
M^*_{n,k} \leq (n-1) k - (k-1)^2/2 - 2 (k-1) + (n-1)  = n (k+1) - k^2/2 - 2 k + 1/2. \qedhere
\]
\end{proof}

We are now ready for the main construction to show the upper bound on $M^*_{n,k}$. This uses a type of graph embedding in the torus; we describe the construction in more detail in Appendix~\ref{construct}.

\begin{lemma}
  \label{cc-lem1}
  Suppose there exists a graph $T$ embedded in the torus with $r$ faces, where each edge appears in two distinct faces, and each face $F$ has $s_F \geq 4$ edges. Let $g = \sum_{F} s_F$. Then for $k \geq 3$ there is a critical $k$-truss with $r (k-2) + g/2$ vertices and $r \binom{k-1}{2} + (k-1 + 1/2) g$ edges.
\end{lemma}
\begin{proof}  
  For each face $F$, let us define $T(F)$ to be the corresponding subgraph of $T$.  We form the graph $G$ by starting with $T$. For each face $F$, we insert a copy of $K_{k-1}$, which we denote by $C(F)$. We add an edge from every vertex of $C(F)$ to every vertex in $T(F)$, and we let $H(F)$ denote these edges. We also define $G(F) = C(F) \cup H(F) \cup T(F)$. See Figure~\ref{fig3}. 
  
  \vspace{0.2in}
\begin{figure}[H]
\begin{center}
  \includegraphics[trim = 2cm -2cm 3cm 2cm,scale=0.4,angle = 0]{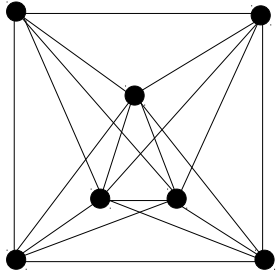}
  
\vspace{-3.2in}
\caption{This shows $G(F)$ for a single face $F$ (here, the outside square), for $k = 4$. The triangle inside the square is $C(F)$. The square on the outside is $T(F)$. Each vertex of $C(F)$ is connected to each vertex of  face $T(F)$, via an edge of $H(F)$.
  \label{fig3}}
\end{center}
\end{figure}
\vspace{-0.2in}

We first compute the number of vertices and edges in $G$. First, since each edge of $T$ appears in exactly two faces, $T$ has $\tfrac{1}{2} \sum_{F} s_F = g/2$ edges. By Euler's formula, $T$ therefore has $g/2 - r$ vertices. Each face $F$ of $T$ gives $k-1$ vertices and $\binom{k-1}{2}$ edges in $C(F)$ and $(k-1) s_F$ edges in $H(F)$. So $G$ has $(g/2 - r)  + r (k-1) = r (k-2) + g/2$ vertices and has $g/2 + \sum_{F} \bigl( \tbinom{k-1}{2} + (k-1) s_F \bigr) = r \tbinom{k-1}{2} + (k-1 + 1/2) g$ edges, as we have claimed.

  Let us check that $G$ is a $k$-truss. For an edge $e$ of $T$, the two corresponding faces include copies of $K_{k-1}$, so $e$ has at least $2 k - 2 \geq k$ triangles. An edge of $C(F)$ has $k-3$ triangles within $C(F)$ and at least $s_F \geq 4$ triangles from vertices of $T(F)$, for a total of $k+1$ triangles. An edge of $H(F)$ has $2$ triangles in $T(F)$ and $k-2$ triangles within $C(F)$, a total of $k$ triangles.  
  
  We next need to show that $G$ is critical. Suppose that $G(L)$ is a $k$-truss for $L \subseteq E$; we need to show that $L = E$ or $L = \emptyset$. We do this in four stages.
  \begin{enumerate}
  \item[(a)]   We first claim that, for every face $F$, either $L$ contains all the edges in $H(F)$, or none of them. For, suppose $L$ omits an edge $e = (u,v) \in H(F)$, where $u \in T(F)$ and $v \in C(F)$. Every other edge $e'$ in $H(F)$ incident to $u$ would then have at most $k-1$ triangles in $L$, and so also $e' \notin L$. Thus, $L$ contains no edges of $H(F)$ incident to $u$.    

    Next, consider a vertex $u'$ adjacent to $u$ in $T(F)$. Any edge $e' \in H(F)$ incident on $u'$ can now have at most $k-1$ triangles in $L$, since one of its triangles in $G(F)$ used an edge of $H(F)$ incident on $u$. Thus, $L$ must omit all the edges of $H(F)$ incident on $u'$ as well.

    Continuing this way around the cycle $T(F)$, we see that $H(F) \cap L = \emptyset$.
    
  \item[(b)] We next claim that for every face $F$, if $L$ omits any edge $e \in G(F)$ then it omits all the edges in $H(F)$. For, note that this edge $e$ participates in a triangle with some edge $e' \in H(F)$. Thus $e'$ has at most $k-1$ triangles in $L$ and must be omitted. By part (a), this implies that all the edges in $H(F)$ are omitted.

  \item[(c)]  We next claim that for every face $F$, either $L$ contains all the edges in $G(F)$, or none of them.  From parts (a) and (b), we see that if $L$ omits any such edge, then it omits all the edges in $H(F)$. This implies that each edge $e \in C(F)$ has at most $k-2$ triangles in $L$, and so $e \notin L$. Similarly, each edge $e \in T(F)$ has at most $k-1$ triangles in $L$, coming from the graph $C(F')$ where $F'$ is the other face touching $e$; thus also $e \notin L$.

\item [(d)] Finally, we claim that either $L = E$ or $L = \emptyset$. For, suppose that $G$ omits an edge $e \in G(F)$ for some face $F$. By part (c), $L$ omits every edge in $G(F)$.   Now note that if $F'$ touches $F$ in $T$, then $G(F')$ omits an edge, namely, the common edge of $F$ and $F'$.  By part (c) this implies $G(F') \cap L = \emptyset$. Since $T$ is connected, continuing this way around $T$ we see that $L = \emptyset$. \qedhere
  \end{enumerate}
  \end{proof}

We now get our final estimate for $M^*_{n,k}$:
\begin{theorem}
  \label{mnk-thm}
  For $k \geq 2$ and $n \geq k+4$, we have
  $$
  n (k/2 + 1) \leq M^*_{n,k} \leq n (k/2 + 5/2 - 1/k) + O(k^2) \leq n k/2 + O(n + k^2).
  $$
\end{theorem}
\begin{proof}
  The lower bound is an immediate consequence of Observation~\ref{oo1}.  Lemma~\ref{hh1} shows that $M^*_{n,k} \leq n(k+1) - (k^2/2 + 2 k - 1/2)$; when $n \leq 2 k$, this quantity is $O(k^2)$ and we are done.  Similarly, Theorem~\ref{mstar2} already shows this result when $k = 2$.

  So suppose that $k \geq 3$ and $n > 2 k$, and we want to show the upper bound on $M^*_{n,k}$.  We write $n = i k + j$, where $j \in \{0, \dots, k-1 \}$ and $i \geq 2$.  As we show in Lemma~\ref{construct-lemma}, for these parameters there is a toroidal embedding satisfying the conditions of Lemma~\ref{cc-lem1}, whose faces consist of two $(j+4)$-cycles and $i-2$ four-cycles.

The sum of edge counts for this embedding is given by $\sum_F s_F  =4 (i-2) + 2 (j+4) = 4 i + 2 j$.  By Lemma~\ref{cc-lem1}, there is a critical $k$-truss with $i k + j = n$ vertices and $m =  i \binom{k-1}{2} + (k-1 + 1/2) (4 i + 2 j)$ edges. As $j \leq k - 1$, we have
  \[
  m = \frac{(3 k-1)(k-2) j+(k (k+5)-2) n}{2 k} \leq (k/2 + 5/2 - 1/k) n + O(k^2). \qedhere
  \]
\end{proof}

\section{Practical truss decomposition algorithm}
\label{sec:algorithm1}
We now present Algorithm~\ref{alg:ktruss_1} to compute the trussness $\tau(e)$ of every edge $e$ in a graph. Recall that $\tau(e)$ is the maximal value $k$ such that $e$ is in a $k$-truss and that, after this has been computed, we can easily compute the $k$-truss-components of $G$ by  discarding all edges $e$ with $\tau(e) < k$ and finding the connected components of the resulting graph. 

Our algorithm here is inspired by the Wang \& Cheng \cite{bib:wangcheng2012} and Huang et al. \cite{huang2015approximate} algorithms, but uses simpler data structures. To explain, let us provide a brief summary of their algorithms. At each stage, they find an edge with the fewest incident triangles in the remaining graph, remove this edge, and then update the triangle counts for all neighboring edges. If an edge $e$ has $k$ triangles when it is removed, then it has $\tau(e) = k+1$. This process continues until all the edges have been removed from the graph.

While this is conceptually simple, it can be cumbersome to implement in practice. In particular, this requires relatively heavy-weight data structures to maintain the edges sorted in increasing order of  triangle counts. For example, Wang \& Cheng use a method of \cite{bib:batagelj2003} based on a four-level hierarchy of associative arrays. (See \cite{chiba} for related data structures.) While this could certainly be implemented, it is also clearly more complex than primitive data structures such as arrays.

The key idea of our new algorithm is that, instead of sorting the edges by triangle count, we only maintain an unordered list of edges whose triangle count is below a given threshold. We can afford to periodically re-scan the graph for edges with few triangles. 

Let us begin by recalling the standard simple algorithm to \emph{enumerate} the triangles in $G$:
\begin{algorithm}[H]
\caption{\label{alg:triangle}}
\begin{algorithmic}[1]
  \For{\textbf{each} edge $e = (u,v) \in E$ with $d(u) \leq d(v)$}
\ForAll{$w \in N(u)$}
\State \textbf{if}{$( v, w ) \in E$ and $\text{ID}(w) > \max(\text{ID}(u), \text{ID}(v))$} \textbf{then} Output triangle $(u,v,w)$
  \EndFor
  \EndFor  
\end{algorithmic}
\end{algorithm}

The following bound is immediate:
\begin{observation}
  \label{output-triangle-prop}
  Algorithm~\ref{alg:triangle} runs in $O(m \bar \delta(G))$ time and outputs every triangle exactly once.
\end{observation}

Instead of statically listing triangles, as in Algorithm~\ref{alg:triangle}, our truss decomposition Algorithm~\ref{alg:ktruss_1}  keeps track of them as edges are removed from the graph. The main data structure is the array $\triangle$, which stores the number of triangles in the residual graph containing any given edge $e$. We also use a sentinel value  denoted $\varnothing$ to indicate that edge $e$ is no longer present in the residual graph. Other data structures include a stack $S$ and a list $L$ of edges which need to be processed.

\begin{algorithm}[H]
\caption{\label{alg:ktruss_1}}
\begin{algorithmic}[1]
\State Initialize an empty stack $S$, and initialize an edge-list $L = E$
\State Using Algorithm~\ref{alg:triangle}, compute the triangle counts $\triangle(e)$ for all $e \in E$
\For{$k=1, \dots, \sqrt{2 m}$}
\ForAll{edges $e \in L$}
\State\textbf{if} $\triangle(e) = k-1$ \textbf{then} push $e$ onto $S$
\State\textbf{if} $\triangle(e) = \varnothing$ \textbf{then} remove $e$ from $L$
\EndFor
\While{$S$ is non-empty}
\State Pop edge $e$ from $S$. Let $e = (u,v)$ such that $d(u) \leq d(v)$
\State $\triangle(e) \leftarrow \varnothing$
\ForAll{$w \in N(u)$}
\If{$( v, w ) \in E$ and $\triangle(u,w) \neq \varnothing$ and $\triangle(v,w) \neq \varnothing$}
\State $\triangle( u, w ) \gets \triangle ( u,w ) - 1$. \textbf{if} $\triangle( u,w ) = k-1$ \textbf{then} push $( u,w )$ onto $S$
\State $\triangle( v,w ) \gets \triangle( v,w ) - 1$. \textbf{if} $\triangle( v,w ) = k-1$ \textbf{then} push $( v,w )$ onto $S$
\EndIf
\EndFor
\State Output $\tau(e) = k-1$
\EndWhile
\EndFor
\end{algorithmic}
\end{algorithm}

We refer to each iteration of the loop at line (3) of Algorithm~\ref{alg:ktruss_1} as \emph{round $k$}. 

Let us first remark on the implementation of $L$. At first glance, it would appear to require a linked list, since we need to remove edges $e$ from $L$ in line (6). However, we only delete elements while we are iterating over the entire list, and so we can instead store $L$ as a simple array. When we want to delete $e$ from $L$, we just swap it to the end of the buffer instead.

We also note that round $k = 1$ can be simplified: for an edge $e$ with $\triangle(e) = 0$, we can immediately output $\tau(e) = 0$ and we do not need to push $e$ onto the stack. This optimization can be useful for graphs which have a relatively small number of triangles.

We now show Algorithm~\ref{alg:ktruss_1} has the claimed complexity and correctly computes the values $\tau(e)$. At any given point in the algorithm, we define the set of edges $e \in E$ with $\triangle(e) \neq \varnothing$ as the \emph{residual edges} and denote them by $R$.

\begin{proposition}
  \label{unique-edge-prop}
  Any edge $e$ gets added to $S$ at most once over the entire lifetime of Algorithm~\ref{alg:ktruss_1}.
\end{proposition}
\begin{proof}
  If $e$ gets added to $S$ in round $k$, then line (10) ensures that $\triangle(e) = \varnothing$ by the end of round $k$, so that $e$ never gets added in subsequent rounds. Also, an edge $e$ can be added to $S$ at most once in a given round, since before adding $e$ to $S$ in line (13) or (14) we first decrement $\triangle(e)$.
  \end{proof}
  
\begin{proposition}
  \label{loop-invariant-prop}
Algorithm~\ref{alg:ktruss_1} maintains the following loop invariants on the data structures:
  \begin{enumerate}
  \item[(a)] The array $\triangle$ correctly records triangle counts for the graph $G(R)$.
  \item[(b)] For every edge $e \in R$, either  $\triangle(e) \geq k$ or $e \in S$.
    \item[(c)] Every edge $e \in S$ satisfies $e \in R$ and $\triangle(e) < k$.
  \end{enumerate}
\end{proposition}
\begin{proof}
  Since $R = E$ initially, these properties are satisfied  at line (2). Now suppose these properties are satisfied at the end of round $k - 1$; we want to show they remain satisfied in round $k$ as well.

  From property (b), we know that $\triangle(e) \geq k-1$ for all edges $e \in R$ at the beginning of round $k$. Lines (4) --- (7) maintain the properties. Now, consider the state just before line (9), where we are processing some edge $e \in S$. Property (c) is maintained since  $e$ does not appear elsewhere in $S$. For property (a), since $e$ gets removed from $R$ at line (10), the counts in $\triangle$ must be updated for all triangles of $R$ involving $e$. Line (13) --- (14) are reached for each such triangle and $\triangle(f)$ is indeed properly adjusted for each edge $f$ neighboring $e$.  Finally for property (b), note that if any such edge $f$ was placed into $S$, it would necessarily have $f \in R$ and $\triangle(f) = k-1$.  
  \end{proof} 

\begin{proposition}
  \label{vn1}
An edge $e$ is in $R$ at the end of round $k$ if and only if $\tau(e) \geq k$.
\end{proposition}
\begin{proof}
By Proposition~\ref{loop-invariant-prop}(b), every edge $f \in R$ has either $\triangle(f) \geq k$ or $f \in S$. Since $S$ is empty after line (18), this implies that each  $e \in R$ has at least $k$ triangles in $G(R)$, and hence $\tau(e) \geq k$.

  Conversely, suppose that some edge $e$ with $\tau(e) \geq k$ gets removed from $R$ before round $k+1$. Let $f$ be the first such edge removed and let $E'$ denote the $k$-truss-component containing $f$. Consider the state before line (9) when $f$ is removed. Since $f$ is the first such removed edge, all the other edges in $E'$ remain in $R$ and so $f$ still has at least $k$ triangles in $R$. This contradicts Proposition~\ref{loop-invariant-prop}(c).
\end{proof}

\begin{theorem}
  \label{correct-thm}
  Algorithm~\ref{alg:ktruss_1} correctly computes $\tau(e)$ for all edges $e$.
\end{theorem}
\begin{proof}
  Suppose that line (17) outputs $\tau(e) = k-1$ for some edge $e$. By Proposition~\ref{loop-invariant-prop}(c), this implies that $e$ must have been in $R$ at the beginning of round $k$, and hence by Proposition~\ref{vn1} we have $\tau(e) \geq k-1$. On the other hand, $e$ got removed from $R$ at line (9), and so by Proposition~\ref{vn1} we have $\tau(e) < k$. Thus, $\tau(e)$ is indeed $k-1$.
  
Note that the termination condition $k = \sqrt{2 m}$ of line (3) follows from Observation~\ref{tt-bcor}.
  \end{proof}

\begin{theorem}
Algorithm~\ref{alg:ktruss_1} runs in $O(m \bar \delta(G))$ time and $O(m)$ memory.
\end{theorem}
\begin{proof}
The array $\triangle$ can be referenced or updated in $O(1)$ time. Bearing in mind our remarks about implementing $L$ as an array, all operations on $S$ and $L$ take $O(1)$ time. Observation~\ref{output-triangle-prop} shows that line (2) takes $O(m \bar \delta(G))$ time and $O(m)$ memory. The data structures $\triangle, S, L$ are indexed by edges, and so overall take $O(m)$ memory.

Next, let $L_k$ denote the value of list $L$ at the beginning of round $k$. By Proposition~\ref{vn1}, $L_k$ consists solely of edges with $\tau(e) \geq k-1$. The runtime of the loop at lines (4) --- (7) at round $k$ is linear in the length of $L_k$, and so the total work over all rounds  is a constant factor times
$$
\sum_k |L_k| \leq \sum_k \bigl | \{ e \in E: \tau(e) \geq k-1 \} \bigr | \leq \sum_e (\tau(e) + 1).
$$

For any edge $e = (u,v)$ with $\tau(e) = k$, we have $\triangle(e) \geq k$ so clearly $d(u) \geq k$ and $d(v) \geq k$. Thus $\tau(e) \leq k \leq \min(d(u), d(v))$, and so
$$
\sum_k |L_k| \leq \sum_e 1 + \min(d(u), d(v)) \leq m + m \bar \delta(G) \leq O(m \bar \delta(G)).
$$

Finally, consider lines  (8) --- (18).  By Proposition~\ref{loop-invariant-prop}(c), every edge $e = (u,v)$ appears at most once in $S$. The enumeration in line (11) takes $O(\min(d(u), d(v)))$ time. So again, the total work for this loop is $\sum_{(u,v) \in E} \min( d(u), d(v)) = m \bar \delta(G)$.
\end{proof}

\section{Truncated truss decomposition using matrix multiplication}
\label{sec:algorithm2}
We now develop a theoretically more efficient algorithm for truncated
truss decomposition up to some given bound $k_{\text{trunc}}$. This allows us to compute the $k$-truss-components for any value $k \leq k_{\text{trunc}}$, by discarding edges with $\tau(e) < k$ and running depth-first search on the resulting graph. 

The algorithm here, like Algorithm~\ref{alg:ktruss_1}, is based on removing edges and updating triangle counts. The crux of the algorithm is the following observation. When an edge is removed, it must be incident on fewer than  $k$ triangles in the residual graph. If we could enumerate these triangles efficiently, then we could potentially update their edges in only $O(k)$ time.

There is a long history of fast matrix multiplication for
triangle enumeraton \cite{bib:ayz1997, bib:yuster2005, bib:bjorklund2014}. These prior algorithms are inherently static: they treat the graph $G$ as a fixed input,  and the output is the
triangle count or list of triangles. To compute the values $\tau(e)$, by contrast, we must  \emph{dynamically}
maintain the triangle lists as edges are removed. We develop an algorithm based
on a methods  of \cite{bib:gkasieniec2009, bib:bjorklund2014} which reduce triangle enumeration to finding witnesses for boolean
matrix multiplication. 

Our algorithm uses  multiplication of rectangular matrices; see \cite{bib:lotti1983} for further details. To measure the cost of this operation, we define the function $\Gamma: [0, 1] \rightarrow \mathbb R_+$ as:
\begin{align*}
\Gamma(s) &= \inf \{ p \in \mathbb R : \exists \text{\ \ an $O(t^p)$-time algorithm for $t \times \lceil t^s \rceil$ by $\lceil t^s \rceil \times
  t$ matrix multiplication} \}
\end{align*}

It is shown in \cite{bib:legall2018} that $\Gamma(s) = 2$ for $s \leq 0.31$, and it is conjectured that $\Gamma(s) = 2$ for all $s \in [0,1]$. The value $\Gamma(1)$ is also known as the \emph{linear-algebra constant $\omega$}. By standard reductions, there is a single randomized algorithm to multiply $t \times \lceil t^s \rceil$ by $\lceil t^s \rceil \times t$ matrices in time $t^{\Gamma(s) + o(1)}$ for all $s$.

\subsection{Algorithm description}
We begin by generating $L = 10 k_{\text{trunc}} \log n$ random vertex sets $X_1, \dots, X_L$, where each vertex goes into each $X_i$ independently with probability $q = \frac{1}{k_{\text{trunc}}}$. We also define $Y_v = \{ \ell \in [L]: v \in X_\ell \}$ for each vertex $v$. The algorithm maintains two data structures corresponding to these sets, keeping track of the following information for each edge $e = (u,v)$:
\begin{enumerate}
\item The total triangle count $\triangle(e)$
\item $S( e,\ell) = \sum_{ w \in N(u) \cap N(v) \cap X_{\ell}} \text{ID}(w)$ for each $\ell \in [L]$
\end{enumerate}

An outline is shown in Algorithm~\ref{alg:ktruss_3}. We provide more detail below on the implementation and runtimes of the steps.   The proof of correctness is essentially the same as Theorem~\ref{correct-thm}, so we do not provide it here.

\begin{algorithm}[H]
\caption{\label{alg:ktruss_3}}
\begin{algorithmic}[1]
\State Generate the random vertex subsets $X_1, \dots, X_L \subseteq V$ and corresponding sets $Y_v$
\State Compute $S(e,\ell)$ and $\triangle(e)$ for every edge $e$ and every $\ell \in [L]$
\For{$k = 1, \dots, k_{\text{trunc}}$}
\While{$\triangle(e) < k$ for some edge $e \in G$}
\State Output $\tau(e) = k-1$
\State Enumerate all the triangles involving $e$ in the residual graph $G$
\State Remove $e$ from $G$ and appropriately update $S, \triangle$
\EndWhile
\EndFor
\State \textbf{for all} remaining edges in $G$ \textbf{do} Output $\tau(e) \geq k_{\text{trunc}}$
\end{algorithmic}
\end{algorithm}

We write $k_{\text{trunc}} = m^a$; since $\tau(e) \leq \sqrt{2 m}$ for all edges $e$, we assume that $k_{\text{trunc}} \leq \sqrt{2 m}$ and hence  $a \in [0, \tfrac{1}{2} + o(1)]$. We also use the notation $\tilde O(x) = x \times \text{polylog}(x)$ for any quantity $x$.

\ \\

\noindent \textbf{Line (1)} With standard sorting methods, this takes $\tilde O(L n)$ time and memory in the worst case.

\noindent \textbf{Line (2)} We use a method based on \cite{bib:yuster2005} for this step. The calculation of $\triangle(e)$ is similar to $S(e,\ell)$ so we only show the latter. 

We divide the vertices into two classes: the \emph{heavy} vertices $v$ (if $d(v) > m^{1-b}$) and the \emph{light} vertices
(if $d(v) \leq m^{1-b}$). Here, $b$ is a parameter in the range $[a,1]$ we will set later. We let $H$ denote the set of heavy vertices, and observe that $|H| \leq 2 m / m^{1-b} = 2 m^b$.

We first compute the contribution to $S$ coming from triangles with at least one light vertex. We do this by looping over light vertices and pairs of their incident edges. For each such triangle $t = (u,v,w)$, we update $S( (u,v), \ell) \leftarrow S( (u,v), \ell) + \text{ID}(w)$ for each $\ell \in Y_w$; we similarly update values $S( (u,w),\ell )$ and $S((v,w), \ell)$. This simple algorithm has expected runtime $m^{2-b + o(1)}$.

We next consider the triangles with only heavy vertices. For each index $\ell$ we form matrices $B$ and $B'$ of dimensions $|H| \times |H \cap X_{\ell}|$ as follows:
\begin{align*}
B(v,w) &= \begin{cases}
1 & \text{if $(v,w) \in E$} \\
0 & \text{if $(v,w) \notin E$} 
\end{cases} &
B'(v,w) &= \begin{cases}
\text{ID}(w) & \text{if $(v,w) \in E$} \\
0 & \text{if $(v,w) \notin E$}
\end{cases} 
\end{align*}

For an edge $e = (u,v)$ on heavy vertices $u,v$ we compute the contribution to $S(e, \ell)$ from heavy vertices $w$ as:
$$
\sum_{\substack{w \in N(u) \cap N(v) \\ w \in H \cap X_{\ell}}} \text{ID}(w) = (B^{\top} B')_{u,v}.
$$

This matrix multiplication $B^{\top} B'$ can be computed in $|H|^{\Gamma\bigl( \frac{ \log |H \cap X_{\ell}|}{\log |H|} \bigr) + o(1)}$ time and $|H|^2 = O(m^{2 b})$ memory. As $|H \cap X_{\ell}|$ is a binomial random variable with mean $q |H| \leq 2 m^{b-a}$, the total expected time for this over all indices $\ell$ is $m^{a + b \Gamma(1 - a/b) + o(1)}$.

\noindent \textbf{Line (4)} We use similar data structures to Algorithm~\ref{alg:ktruss_1} for this. For each round $k$, we check if $\triangle(e) < k$ for each edge $e$. Whenever we process an edge $e$ in lines (4) --- (9), we check if  $\triangle(f) = k-1$ for some neighboring edge $f$ and, if so, add $f$ to a stack.

\noindent \textbf{Line (6)} We use the following primary algorithm to enumerate triangles containing edge $e = (u,v)$: for each $\ell \in [L]$ with $S(e, \ell) \leq n$, let $x_{\ell}$ be the vertex with $\text{ID}(x_{\ell}) = S(e,\ell)$ and test if there are edges $(u,x_{\ell}), (v,x_{\ell})$ in  the residual graph. If so, then output a triangle $(u, v, x_{\ell})$.

Let us define $A$ to be the set of vertices $w$  where the edges $(u,w)$ and $(v,w)$ remain in the current residual graph, and let $W \subseteq A$ denote the set of all triangle vertices $x_{\ell}$ enumerated in the primary algorithm. Note that we also maintain the residual triangle count $\triangle(e) = |A|$. If $|W|$ is equal to the known value $\triangle(e)$, then we have found all the triangles containing $e$. If $|W| < \triangle(e)$, then use a second, slower, fallback option:  we simply loop over all vertices in $V$.

We argue now that $|W| = \triangle(e)$ with high probability. Note first that the update sequence in Algorithm~\ref{alg:ktruss_3} is not affected by the randomness in the sets $X_\ell$. So $A$ can be regarded as a deterministic quantity. For an arbitrary vertex $w \in A$, observe that if $A \cap X_{\ell} = \{w \}$ for some index $\ell$, then $S(e, \ell) = \text{ID}(w)$ and hence  $w$ will go into $W$. For each $\ell$, we have $A \cap X_{\ell} = \{ w \}$ with probability $q (1 - q)^{|A| - 1}$. As $|A|  \leq k \leq k_{\text{trunc}}$ and $q = \tfrac{1}{k_{\text{trunc}}}$ and $L = 10 k_{\text{trunc}} \log n$, this shows that
$$
\Pr(w \notin W) \leq (1 - q (1 - q)^{k_{\text{trunc}}})^L \leq n^{-3.6}.
$$

By a union bound over $w \in A$, this implies that $\Pr( W = A) \geq 1 -  n^{-2.6}$. The primary algorithm of generating the set $W$ takes $O(L) = O(k_{\text{trunc}} \log n)$ time. The fallback algorithm takes $O(n)$ time in the worst case, so it contributes expected runtime at most $n^{-2.6} \times O(n) = o(1)$.

\noindent \textbf{Line (7)} Suppose we are removing edge $e = (u,v)$, and we have enumerated all the triangles containing $e$ in the residual graph.  For each such triangle $(u,v,w)$, we update $\triangle$ by decrementing $\triangle( u, w)$ and $\triangle(v, w)$. We update $S$ by setting $S((u,w), \ell) \leftarrow S((u,w),\ell) - \text{ID}(v)$ for each $\ell \in Y_v$ and $S((v,w), \ell) \leftarrow S((v,w), \ell) - \text{ID}(u)$ for each $\ell \in Y_u$.

Since there are at most $k \leq k_{\text{trunc}}$ triangles, and since the sets $Y_u$ and $Y_v$ have expected size $O(\log n)$, the overall expected time for line (8) is $O(k_{\text{trunc}} \log n)$.

\subsection{Putting it together}
We can now calculate the overall complexity of the algorithm.
\begin{theorem}
  \label{thm1}
  For any real number $r$ satisfying $\Gamma\bigl( \tfrac{2 - a(2+r)}{2-a} \bigr) \leq r \leq 2/a - 2$,   Algorithm~\ref{alg:ktruss_3} can be implemented to run in $m^{\frac{2 r +  a}{r+1}+o(1)}$ expected time and $\tilde O(m^{\frac{4 - 2 a}{r+1}} + m^{1+a})$ memory.
\end{theorem}
\begin{proof}
  The algorithm uses $\tilde O(m^{1+a})$ memory to store the arrays $S(e, \ell)$ and $\triangle(e)$. Line (1) takes $\tilde O(L n)$ time, which is at most $m^{1+a+o(1)}$.

  Lines (5) --- (8) are executed at most once per edge, so their total expected time is at most $\tilde O(m k_{\text{trunc}}) \leq m^{1+a+o(1)}$.  Since $r \geq \Gamma\bigl( \tfrac{2 - a(2+r)}{2-a} \bigr) \geq 2$ and $a \leq \tfrac{1}{2} + o(1)$, this is at most $m^{\frac{2 r+a}{r+1} + o(1)}$.
  
  For line (2), we set $b = \frac{2-a}{r+1}$, which is in the range $a \leq b \leq 1$.  This uses $\tilde O(m^{\frac{4-2a}{r+1}})$ memory. The runtime is $m^{2-b+o(1)} + m^{a + b \Gamma(1 - a/b) + o(1)}$, which by our bound on $r$ is at most $m^{\frac{2 r + a}{r+1} + o(1)}$. 
\end{proof}

The current known
 bounds on $\Gamma(s)$ are highly non-linear functions of $s$. Thus, the parameters for Theorem~\ref{thm1} cannot be optimized in closed form. We can obtain slightly crude estimates in terms of $\omega$ or using analysis of $\Gamma$ from \cite{bib:legall2018}. We focus on the case where $k_{\text{trunc}}$ is small; note that even $k_{\text{trunc}} = O(1)$ may be relevant for cohesive subnetworks of real-world graphs.

\begin{theorem}
  \label{pp1}
  \begin{enumerate}
  \item Algorithm~\ref{alg:ktruss_3} runs in $m^{\frac{2 \omega + a}{\omega+1} + o(1)}$  time and $\tilde O(m^{\frac{4 - 2a}{\omega+1}} + m^{1+a})$ memory.
  \item If $ \omega = 2$, Algorithm~\ref{alg:ktruss_3} runs in  $m^{4/3 + o(1)} k_{\text{trunc}}^{1/3}$  time and $\tilde O(m^{4/3} + m k_{\text{trunc}})$ memory.
    \item For $k_{\text{trunc}} \leq m^{0.029}$, Algorithm~\ref{alg:ktruss_3} runs in
      $O(m^{1.4071} k_{\text{trunc}}^{0.0648})$  time and $O(m^{1.1860})$ memory.
      \end{enumerate}
  \end{theorem}
\begin{proof}
  \begin{enumerate}
  \item If $a \leq 0.4$, then  apply Theorem~\ref{thm1} with $r = \omega$;
    note then that $2/a - 2 \geq 3 \geq r$ and $\Gamma( \frac{2 - a(2+r)}{2-a}) \leq \omega = r$ as required. This runs in the claimed time and memory.

    If $a > 0.4$, then  apply Theorem~\ref{thm1} with $r = 2$. Note then that $2/a - 2 \geq 2 = r$ and $\Gamma( \frac{2 - a(2+r)}{2-a} ) \leq \Gamma( 0.25) = 2 = r$. Algorithm~\ref{alg:ktruss_3} runs in $m^{\frac{4 +  a}{3}+o(1)}$  time and $\tilde O(m^{\frac{4 - 2 a}{3}} + m^{1+a})$ memory. Since $\omega \geq 2$, this is also in the required range.

  \item This follows immediately from the preceding paragraphs, setting $\omega = 2$.

    \item  From \cite{bib:legall2018} we see that $\Gamma(0.95) \leq
      2.333789$ and from \cite{bib:legall2014} we see that $\omega = \Gamma(1) \leq 2.3728639$. As the function $\Gamma$ is concave-up \cite{bib:lotti1983}, this implies that when $0.95 \leq x \leq 1$ we have
      \begin{equation}
\label{bb1}
\Gamma(x) \leq 2.333789 + 0.781498 (x - 0.95)
\end{equation}

Set $r = 14.43632 -\frac{110.419519}{9.15323 - a}$ and let $s = \frac{2 -
  a(2+r)}{2 - a}$. We have $s \geq 0.95$ for $a \leq 0.029$, and from Eq.~(\ref{bb1}) we can verify that $\Gamma(s) \leq r \leq 2/a - 2$ for $a$ in this range. Further calculus shows that then $\frac{2 r + a}{r+1} \leq 1.40704 + 0.0648 a$ so the  runtime is $O(m^{1.4071} k_{\text{trunc}}^{0.0648})$.  The memory is $\tilde O(m^{1+a} + m^{\frac{4 - 2 a}{r+1}})$, which in this range is $O(m^{1.1860})$. \qedhere
\end{enumerate}
\end{proof}

\section{Acknowledgments}
Thanks to Michael Murphy, Noah Streib, Lowell Adams, Tad White and Randy Dougherty for ideas and discussion. Thanks to the anonymous journal and conference reviewers for helpful suggestions and pointing us to some useful references.

\appendix

\section{Properties of degeneracy and average degeneracy}
\label{app:arb}
The \emph{degeneracy} of graph $G$, denoted $\delta(G)$, is the minimum value such that for every vertex set $U \subseteq V$, there is a vertex in $G[U]$ of degree at most $\delta(G)$. Many graph classes have bounded degeneracy, for example, planar graphs have degeneracy at most $5$.

We remark that there is a connection between degeneracy and graph trussness:
\begin{proposition}
  \label{tau-bnd}
For any edge $e$ we have $\tau(e) \leq \delta(G) - 1$.
\end{proposition}
\begin{proof}
  Consider any $k$-truss-component $E'$, and let $U$ denote the set of vertices with at least one edge in $E'$. Then $G[U]$ contains $G(E')$. By Observation~\ref{obs:mintriangle}, every vertex in $G(E')$ has degree at least $k+1$. So $G[U]$ has minimum degree at least $k+1$, which shows that $\delta(G) \geq k+1$.
\end{proof}

A number of previous triangle and trussness algorithms such as \cite{chiba, huang2015approximate} have analyzed runtime in terms of degeneracy as well as a closely related graph parameter known as \emph{arboricity}, denoted $\alpha(G)$. This is also related to a parameter known as the \emph{pseudo-arboricity} $\alpha^*(G)$. See \cite{picard} for definitions and background. The following are some standard and well-known bounds:
\begin{proposition}
  \label{app-arb1}
  Consider a graph $G = (V,E)$ with $m = |E|$ edges.
  \begin{itemize}
  \item $G$ has an edge-orientation in which every vertex has out-degree at most $\alpha^*(G)$.
  \item  $G$ has an acyclic edge-orientation in which every vertex has out-degree at most $\delta(G)$.
  \item $\alpha(G) \leq \sqrt{2 m}$.
  \item $\alpha^*(G) \leq \alpha(G) \leq \alpha^*(G) + 1$.
    \item $\alpha(G) \leq \delta(G) < 2 \alpha^*(G)$.
    \end{itemize}
\end{proposition}

Note that, in light of the last two bounds in Proposition~\ref{app-arb1}, asymptotic runtime bounds are equivalent for $\alpha, \alpha^*$, and $\delta$.

For any edge $e = (u,v)$, define $h(e) = \min(d(u), d(v))$. With this notation, we recall the definition of average degeneracy as $\bar \delta(G) = \frac{1}{m} \sum_{\text{edges $e$}} h(e)$. The following result shows some intuition behind the term ``average degeneracy.''  To state it informally, ``most'' edges of $G$ have degeneracy ``not much larger'' than $\bar \delta(G)$.
\begin{proposition}
  Consider a graph $G = (V,E)$.
  \begin{enumerate}
  \item The average degeneracy $\bar \delta(G)$ satisfies $\bar \delta(G) \leq 2 \alpha^*(G)$.
  \item For any $x \in (0,1)$, there is an edge-set $L \subseteq E$ with $|L| \geq (1-x) |E|$ and $\delta(G(L)) \leq \bar \delta(G) / x$.
  \end{enumerate} 
\end{proposition}
\begin{proof}
The first result (in a slightly weaker form) was shown by \cite{chiba}; for completeness, we provide a version of their proof here.  By Proposition~\ref{app-arb1} there exists an orientation of $E$ such that every vertex has out-degree at most $\alpha^*(G)$. Now compute:
  \begin{align*}
    \sum_{e \in E} h(e) \leq \negthickspace \negthickspace  \sum_{\substack{e = (u,v) \in E\\ \text{ $e$ oriented to $v$}}} \negthickspace  d(u) = \sum_{u \in V} \text{out-degree}(u) \times d(u) \leq \sum_{u \in V} \alpha^*(G) \times d(u) = 2 m \alpha^*(G)
    \end{align*}

  For the second result, let $L$ denote the set of edges $e$ with $h(e) \leq s = \bar \delta(G)/x$. Since $\sum_{e} h(e) = m \bar \delta(G)$,  we must have $|E - L| s \leq m \bar \delta(G)$, i.e. $|L| \geq (1-x) |E|$.

  Now consider a vertex set $U \subseteq V$ and let $G' = G(L)[U]$. Each edge $e \in G'$ has an endpoint $u \in U$ with $d(u) \leq s$. Thus, the total number of edges in $G'$ is at most $\sum_{u \in U: d(u) \leq s} d(u) \leq |U| s$. Since this holds for all $U$, the graph $G(L)$ has degeneracy at most $s$.
\end{proof}

\section{Constructions of toroidal graph embedding}
\label{construct}
\begin{lemma}
  \label{construct-lemma}
  For any integers $i \geq 0, t \geq 4$, there is a graph embedding in the torus whose faces consist of two $t$-cycles and $i$ four-cycles, and where each edge is in two distinct faces.
\end{lemma}
\begin{proof}
  There are two cases depending on the parity of $t$, as shown in Figure~\ref{fig1}.
  
  \textbf{Case I: $t = 2 s$.} We view each $t$-cycle as a rectangle with side lengths $s-1$ and $1$. They are stacked vertically on top of $i$ unit-length squares. Overall, we have one large rectangle with vertical sides of length $i + 2 (s-1)$ and horizontal sides of length $1$.
  
To get a torus, we first identify the top and bottom edges (marked $X$) to form a cylinder with circular circumference $i + 2 (s-1)$. We next identify the left side of the cylinder with a twisted version of the right side, namely, we rotate one of the sides by $s-1$. Thus, the two edges marked $Y$ are identified.

\textbf{Case II: $t = 2 s + 1$.} We view each $t$-cycle as a trapezoid with base length $1$ and side lengths $s-1$ and $s$, joined to form a rectangle of height $2 s - 1$. We put the $i$ squares below them, and use a similar twisting process to join them into a torus.

\vspace{0.25in}

\begin{figure}[H]
  \begin{minipage}[t]{0.48\textwidth}
      \includegraphics[trim = -4cm -2cm 3.5cm 2cm,scale=0.5,angle = 0]{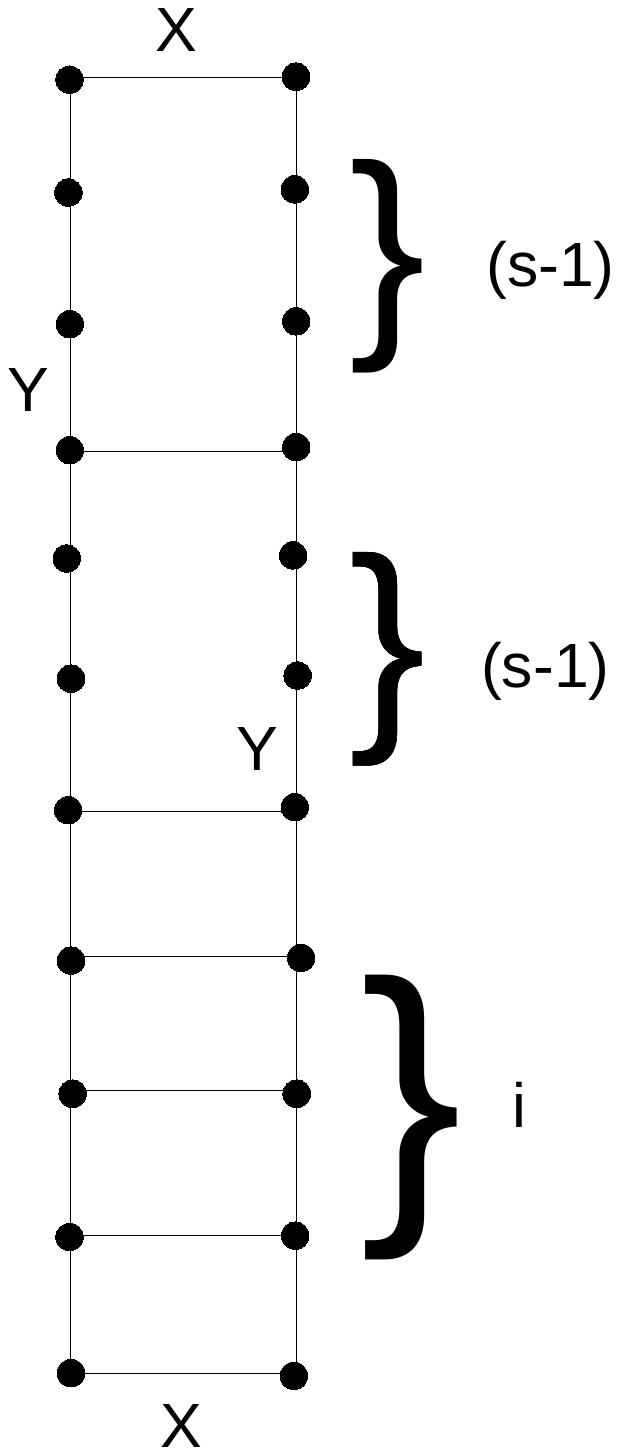}
\end{minipage}
\hspace*{\fill} 
\begin{minipage}[t]{0.48\textwidth}
   \includegraphics[trim = 2cm -1.1cm 3.5cm 4cm,scale=0.5,angle = 0]{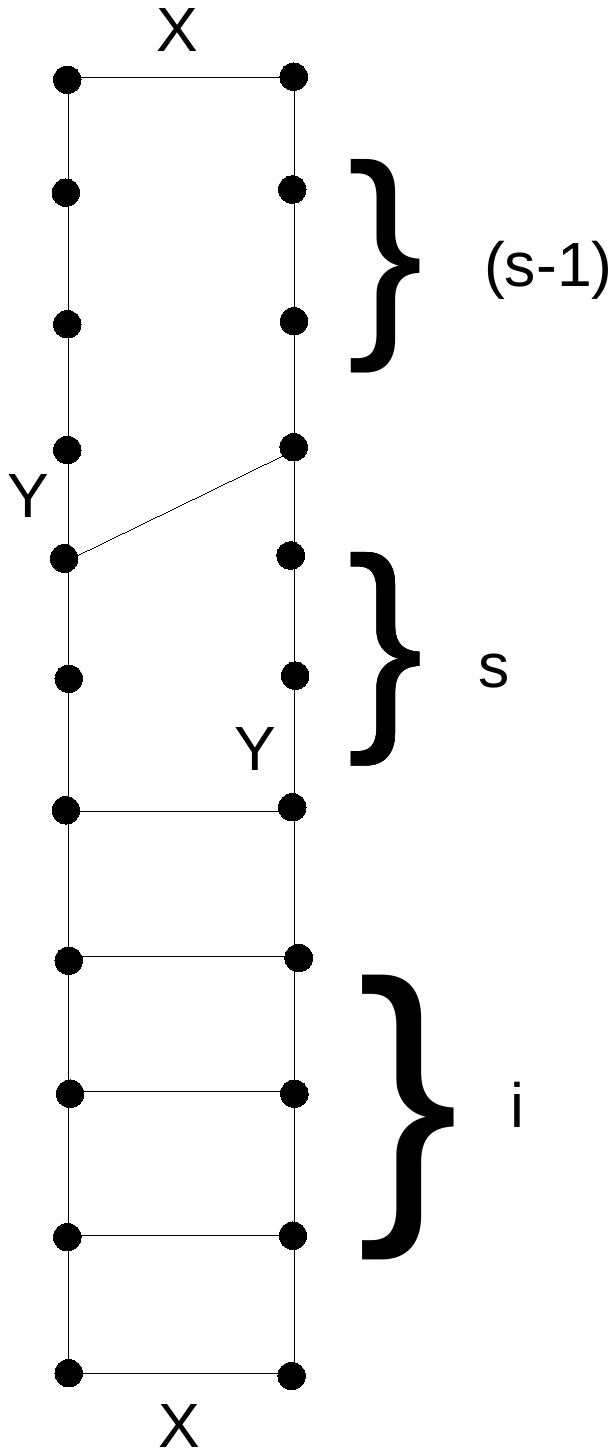}
\end{minipage}
\vspace{-2.7in}
\caption{Two $t$-cycles packed on top of $i$ four-cycles. The case of even values of $t$ is shown on the left side and the case of odd values of $t$ on the right side.}
\label{fig1}
\end{figure}

The horizontal edges clearly have distinct faces. Also, every left edge is identified with a right edge from a distinct face. For example, when $t = 2 s$, the top $s-1$ left edges come from the top $t$-cycle and they are identified with the $s-1$ right edges from the bottom $t$-cycle.
\end{proof}

\bibliographystyle{plain}
\bibliography{ktruss}

\end{document}